\documentclass{amsart}

\usepackage{amsfonts}
\usepackage{amsmath}
\usepackage{amsthm}
\usepackage{graphicx}
\usepackage{wasysym}
\usepackage{amssymb}
\usepackage{mathrsfs}
\usepackage{amscd}
\usepackage{lscape}
\usepackage{longtable}
\usepackage{rotating}
\usepackage{phonetic}
\usepackage{pstricks}
\usepackage{pst-node}
\usepackage[crop=off]{auto-pst-pdf}
\usepackage{etex}
\usepackage{float}

\input prepictex
\input pictex
\input postpictex

\newtheorem{lemma}{Lemma}[section]
\newtheorem{thm}[lemma]{Theorem}

\newtheorem{prop}[lemma]{Proposition}

\newtheorem{conj}[lemma]{Conjecture}

\newtheorem{example}[lemma]{Example}

\newtheorem{claim}[lemma]{Claim}

\title[Frankl-F\"{u}redi-Kalai Inequalities on $\gamma$-vectors]{Frankl-F\"{u}redi-Kalai Inequalities on the $\gamma$-vectors of flag nestohedra}
\author{Natalie Aisbett}
\address{School of Mathematics and Statistics\\
University of Sydney, NSW, 2006\\
Australia}
\email{N.Aisbett@maths.usyd.edu.au}
\date{}

\begin{document}

\maketitle{}

\numberwithin{figure}{section}
\numberwithin{equation}{section}

\begin{abstract}
For any flag nestohedron, we define a flag simplicial complex whose $f$-vector is the $\gamma$-vector of the nestohedron. This proves that the $\gamma$-vector of any flag nestohedron satisfies the Frankl-F\"{u}redi-Kalai inequalities, partially solving a conjecture by Nevo and Petersen \cite{np}. We also compare these complexes to those defined by Nevo and Petersen in \cite{np} for particular flag nestohedra.
\end{abstract}

\begin{section}{Introduction}

For any building set $B$ there is an associated simple polytope $P_B$ called the \emph{nestohedron} (see Section 2 below, \cite[Section 7]{po} and \cite[Section 6]{prw}). When $B = B(G)$ is the building set determined by a graph $G$, $P_{B(G)}$ is the well-known graph-associahedron of $G$ (see \cite[Example 2.1]{ai}, \cite[Sections 7 and 12]{prw}, and \cite{vol}). The numbers of faces of $P_B$ of each dimension are conveniently encapsulated in its $\gamma$-polynomial $\gamma(B) =\gamma(P_B)$ defined below.\\

Recall that for a $d-1$-dimensional simplicial complex $\Delta$, the $f$-polynomial is a polynomial in $\mathbb{Z}[t]$ defined as follows:

$$f(\Delta)(t):= f_0 + f_1t + \cdots +f_{d}t^{d},$$ where $f_i(\Delta)$ is the number of $(i-1)$-dimensional faces of $\Delta$, and $f_0(\Delta) =1$. The $h$-polynomial is given by

$$h(\Delta)(t): =(t-1)^{d}f(\Delta)\left(\frac{1}{t-1}\right).$$  When $\Delta$ is a homology sphere $h(\Delta)$ is symmetric (this is known as the Dehn-Somerville relations) hence it can be written

$$h(\Delta)(t) = \sum_{i=0}^{\lfloor\frac{d}{2}\rfloor}\gamma_it^i(1+t)^{d-2i},$$ for some $\gamma_i \in \mathbb{Z}$. Then the $\gamma$-polynomial is given by
$$\gamma(\Delta)(t): = \gamma_0 + \gamma_1t + \cdots +\gamma_{\lfloor \frac{d}{2}\rfloor}t^{\lfloor \frac{d}{2}\rfloor }.$$

\noindent The vectors of coefficients of the $f$-polynomial, $h$-polynomial and $\gamma$-polynomial are known respectively as the $f$-vector, $h$-vector and $\gamma$-vector. If $P$ is a simple $(d+1)$-dimensional polytope then the dual simplicial complex $\Delta_P$ of $P$ is the boundary complex (of dimension $d$) of the polytope that is polar dual to $P$. The $f$-vector, $h$-vector and $\gamma$-vector of $P$ are defined via $\Delta_P$ as

$$f(P)(t) := t^{d}f(\Delta_P)(t^{-1}),$$ so that $f_i(P)$ is the number of $i$ dimensional faces of $P$, and

$$h(P)(t):=  h(\Delta_P)(t)$$ $$\gamma(P)(t): = \gamma(\Delta_P)(t).$$

When $B$ is a building set, we denote the $\gamma$-polynomial for $P_{B}$ by $\gamma(B)$.\\

Recall that a simplicial complex $\Delta$ is \emph{flag} if every set of pairwise adjacent vertices is a face. Gal conjectured

\begin{conj}
\cite[Conjecture 2.1.7]{gal}. If $\Delta$ is a flag homology sphere then $\gamma(\Delta)$ is non negative.
\end{conj}

This implies that the $\gamma$-vector of any flag polytope has non negative entries. Gal's conjecture was proven for flag nestohedra by Volodin in \cite[Theorem 9]{vol}. Frohmader \cite[Theorem 1.1]{fr} showed that the $f$-vector of any flag simplicial complex satisfies the Frankl-F\"{u}redi-Kalai inequalities. Nevo and Petersen conjectured the following strengthening of Gal's conjecture:

\begin{conj}
\cite[Conjecture 6.3]{np}. If $\Delta$ is a flag homology sphere then $\gamma(\Delta)$ satisfies the Frankl-F\"{u}redi-Kalai inequalities. \label{bigconj}
\end{conj}
They proved this in \cite{np} for the following classes of flag spheres:

\begin{itemize}
\item $\Delta$ is a Coxeter complex (including the simplicial complex dual to $P_{B(K_n)}$),
\item $\Delta$ is the simplicial complex dual to an associahedron ($=P_{B(Path_n)}$),
\item $\Delta$ is the simplicial complex dual to a cyclohedron ($=P_{B(Cyc_n)}$),
\item $\Delta$ has $\gamma_1(\Delta) \le 3$,
\end{itemize}
by showing that the $\gamma$-vector of such $\Delta$ is the $f$-vector of a flag simplicial complex. In \cite{npt} this is proven for the barycentric subdivision of a simplicial sphere.\\

In this paper we prove Conjecture \ref{bigconj} for all flag nestohedra:

\begin{thm}
If $P_B$ is a flag nestohedron, there is a flag simplicial complex $\Gamma(B)$ such that $f(\Gamma(B))=\gamma(P_B)$. In particular $\gamma(P_B)$ satisfies the Frankl-F\"{u}redi-Kalai inequalities. \label{mainthm}
\end{thm}

Our construction for $\Gamma(B)$ depends on the choice of a ``flag ordering" for $B$ (see Section 3 below). In the special cases considered by \cite{np} our $\Gamma(B)$ does not always coincide with the complex they construct.\\

Here is a summary of the contents of this paper. Section 2 contains preliminary definitions and results relating to building sets and nestohedra. In Section 3 we define the flag simplicial complex $\Gamma(B)$ for a building set $B$ and prove Theorem \ref{mainthm}. In Section 4 we compare the simplicial complexes $\Gamma(B)$ to the flag simplicial complexes defined in \cite{np}, and give combinatorial definitions for $\Gamma(B)$ when $B = B(K_n)$ and $B=B(K_{1,n-1})$.\\

\textbf{Acknowledgements}\\

This paper forms part of my PhD research in the School of Mathematics and Statistics at the University of Sydney. I would like to thank my supervisor Anthony Henderson for his feedback and help.

\end{section}

\begin{section}{Preliminaries}

A \emph{building set} $B$ on a finite set $S$ is a set of non empty subsets of $S$ such that
\begin{itemize}
\item For any $I,~ J \in B$ such that $I \cap J \ne \emptyset$, $I \cup J \in B$.
\item $B$ contains the singletons $\{i\}$, for all $i \in S$.
\end{itemize}

$B$ is \emph{connected} if it contains $S$. For any building set $B$, $B_{max}$ denotes the set of maximal elements of $B$ with respect to inclusion. The elements of $B_{max}$ form a disjoint union of $S$, and if $B$ is connected then $B_{max} = \{S\}$. Building sets $B_1$, $B_2$ on $S$ are \emph{equivalent}, denoted $B_1 \cong B_2$, if there is a permutation $\sigma: S \rightarrow S$ that induces a one to one correspondence $B_1 \rightarrow B_2$. \\

Let $B$ be a building set on $S$ and $I \subseteq S$. The \emph{restriction of $B$ to $I$} is the building set $$B|_{I}: = \{b ~|~ b \in B, ~b \subseteq I \}~~\hbox{on $I$}.$$ The \emph{contraction of $B$ by $I$} is the building set
$$B/I: = \{b \backslash I~|~b \in B,~b \not \subseteq I\}~~\hbox{on $S-I$.}$$

We associate a polytope to a building set as follows. Let $e_1,....,e_n$ denote the standard basis vectors in $\mathbb{R}^n$. Given $I \subseteq [n]$, define the simplex $\Delta_I : = ConvexHull(e_i~|~i \in I)$. Let $B$ be a building set on $[n]$. The \emph{nestohedron} $P_B$ is a polytope defined in \cite{po} and \cite{prw} as the Minkowski sum:

$$P_B := \sum_{I \in B}\Delta_I.$$

A $(d-1)$-dimensional face of a $d$-dimensional polytope is called a \emph{facet}. A simple polytope $P$ is \emph{flag} if any collection of pairwise intersecting facets has non empty intersection, i.e. its dual simplicial complex is flag. We use the abbreviation \emph{flag complex} in place of flag simplicial complex. A building set $B$ is \emph{flag} if $P_B$ is flag.\\

A \emph{minimal flag building set} $D$ on a set $S$ is a connected building set on $S$ that is flag, such that no proper subset of its elements forms a connected flag building set on $S$. Minimal flag building sets are described in detail in \cite[Section 7.2]{prw}. They are in bijection with binary trees with leaf set $S$. Given such a tree, the corresponding minimal flag building set consists of the sets of descendants of the vertices of the tree. If $D$ is a minimal flag building set then $\gamma(D) = 1$ (see \cite[Section 7.2]{prw}).\\

Let $B$ be a building set. A \emph{binary decomposition} or \emph{decomposition} of a non singleton element $b \in B$ is a set $D \subseteq B$ that forms a minimal flag building set on $b$. Suppose that $b \in B$ has a binary decomposition $D$. The two maximal elements $d_1,~d_2 \in D-\{b\}$ with respect to inclusion are the \emph{maximal components} of $b$ in $D$. Propositions \ref{flaglem} and \ref{flagprop} give alternative characterizations of when a building set is flag.

\begin{prop}
\cite[Lemma 7.2]{ai}. A building set $B$ is flag if and only if every non singleton $b \in B$ has a binary decomposition. \label{flaglem}
\end{prop}

\begin{prop}
\cite[Corollary 2.6 ]{ai}. A building set $B$ is flag if and only if for every non singleton $b \in B$, there exist two elements $d_1,~d_2 \in B$ such that $d_1 \cap d_2 = \emptyset$ and $d_1 \cup d_2 =b.$ \label{flagprop}
\end{prop}

It follows from Proposition \ref{flagprop} that a graphical building set is flag.

\begin{lemma}
\cite[Lemma 2.7]{ai}. Suppose $B$ is a flag building set. If $a,b \in B$ and $a \subsetneq b$, then there is a decomposition of $b$ in $B$ that contains $a$. \label{decomp}
\end{lemma}

Recall the following theorems of Volodin \cite{vol}:

\begin{thm}\cite[Lemma 6]{vol}.
Let $B$ and $B'$ be connected flag building sets on $S$ such that $B \subseteq B'$. Then $B'$ can be obtained from $B$ by successively adding elements so that at each step the set is a flag building set. \label{flagbuild}
\end{thm}

\begin{thm}\cite[Corollary 1]{vol}. See also \cite[Lemma 3.3]{ai}.
If $B'$ is a flag building set on $S$ obtained from a flag building set $B$ on $S$ by adding an element $b$ then

\begin{align*}\gamma(B') =& \gamma(B) + t\gamma(B'|_b)\gamma(B'/b)\\
=& \gamma(B) + t\gamma(B|_b)\gamma(B/b).\end{align*}
\label{vollem}
\end{thm}

\begin{section}{The flag complex $\Gamma(B)$ of a flag building set $B$}

For a building set $B$ with maximal components $B_{max} = \{b_1,...,b_{\alpha}\},$ let $B_i = B|_{b_i}$ for $i = 1,..,\alpha$. Then we have

$$P_{B} = P_{B_1} \times P_{B_2}\times \cdots \times P_{B_{\alpha}}$$

\noindent which implies that if $\gamma(B_i) = f(\Gamma(B_i))$ for some flag complex $\Gamma(B_i)$, then

$$\gamma(B) =\gamma(B_1)\gamma(B_2)\cdots \gamma(B_{\alpha})= f(\Gamma(B_1)*\Gamma(B_2)*\cdots *\Gamma(B_{\alpha})).$$ Hence to prove Theorem \ref{mainthm} we need only consider connected flag building sets. \\

Suppose that $B$ is a connected flag building set on $[n]$, $D$ is a decomposition of $[n]$ in $B$, and $b_1,b_2,...,b_k$ is an ordering of $B-D$, such that $B_j=D\cup\{b_1,b_2,...,b_j\}$ is a flag building set for all $j$. (Such an ordering exists by Theorem \ref{flagbuild}). We call the pair consisting of such a decomposition $D$ and the ordering on $B-D$ a \emph{flag ordering} of $B$, denoted $O$, or $(D,b_1,....,b_k)$. For any $b_j \in B -D$, we say an element in $B_{j-1}$ is \emph{earlier} in the flag ordering than $b_j$, and an element in $B - B_j$ is \emph{later} in the flag ordering than $b_j$.\\

For any $j \in [k]$ define

$$U_j : =\{i~|~i<j, b_i \not \subseteq b_j,~\hbox{there is no}~ b \in B_{i-1}~\hbox{such that}~b \backslash b_j = b_i \backslash b_j\},$$ and
$$V_j: = \{i~|~i<j,~b_i \subseteq b_j, \exists ~b \in B_{i-1} ~\hbox{such that}~b_i \subsetneq b \subsetneq b_j \}.$$ If $i \in U_j \cup V_j$ then we say that $b_i$ is \emph{non degenerate} with respect to $b_j$. If $b_i \in B_{j-1}$ and $i \not \in U_j \cup V_j$ then $b_i$ is \emph{degenerate} with respect to $b_j$. Degenerate elements with respect to $b_j$ that are not contained in $b_j$ are elements that we need not consider as contributing to the building set $B_j/b_j$. The set of degenerate elements with respect to $b_j$ that are subsets of $b_j$, together with $b_j$, forms a decomposition of $b_j$ in $B_j|b_j$.\\

Given a flag building set $B$ with flag ordering $O = (D,b_1,...,b_k)$ define a graph on the vertex set $$V_O=\{v(b_1),...,v(b_k)\},$$ where for any $i<j$, $v(b_i)$ is adjacent to $v(b_j)$ if and only if $i \in U_j \cup V_j$. Then define a flag simplicial complex $\Gamma(O)$ whose faces are the cliques in this graph. If the flag ordering is clear then we denote $\Gamma(O)$ by $\Gamma(B)$. For any $I \subseteq [k]$, we let $\Gamma(O)|_{I}$ denote the induced subcomplex of $\Gamma(O)$ on the vertices $v(b_i)$ for all $i \in I$. \\

\begin{example}
Consider the flag building set $B(Path_5)$ on $[5]$. It has a flag ordering $O$ given by

$$D = \{\{1\},\{2\},\{3\},\{4\},\{5\},[2],[3],[4],[5]\},$$ and

$$b_1 = \{3,4\},~b_2 = \{2,3,4\},~b_3 =\{2,3\},~b_4 =\{2,3,4,5\},~b_5 = \{3,4,5\},~b_6 = \{4,5\}.$$ Then $\Gamma(O)$ has only two edges, namely

$$\{v(b_2),v(b_6)\}~\hbox{and}~\{v(b_3),v(b_4)\}.$$ These are edges because $b_2 =\{2,3,4\}$ is the earliest element which has image $\{2,3\}$ in the contraction by $b_6$, and the element $b_3 = \{2,3\}$ is a subset of $b_2 =\{2,3,4\}$ which is in turn a subset of $b_4$.\\

\end{example}

Now $D/b_k$ is a decomposition of $[n] -b_k$, and we have an induced ordering of $(B/b_k)-(D/b_k)$, where the $i$th element is $b_{u_i}':=b_{u_i} \backslash b_k$ if $u_i$ is the $i$th element of $U_k$ (listed in increasing order). Then for all $i$, $D/b_k \cup \{b_{u_1}',...,b_{u_i}'\}$ is a flag building set. Hence we can also define a flag complex $\Gamma(B/b_k)$. We label the vertices of $\Gamma(B/b_k)$ by $v(b_{u_1}'),v(b_{u_2}'),...,v(b_{u_{|U_k|}}')$. \\

\end{section}

\begin{claim}
Let $B$ be a connected flag building set with flag ordering $(D,b_1,...,b_k)$. For all $b \in B$ let $b' = b \backslash b_k$. If $b' \ne \emptyset$, $j \in U_k$, and $b \in B_{j-1}$ then $b \subseteq b_j$ if and only if $b' \subseteq b_j'$.\label{incprop}
\end{claim}

\begin{proof}
$\Rightarrow:$ It is clear that $b \subseteq b_j$ implies $b' \subseteq b_j'$. \\

$\Leftarrow:$ Suppose for a contradiction that $b' \subseteq b_j'$ and $b \not \subseteq b_j$. Then $b \cap b_j \ne \emptyset$ and $b \cup b_j \ne b_j$, which implies that (since $B_j$ is a building set) $b \cup b_j \in B_{j-1}$. We also have that $(b \cup b_j)' =b_j'$, which implies that $b_j$ is degenerate with respect to $b_k$, a contradiction.
\end{proof}

\begin{prop}
Let $B$ be a connected flag building set with flag ordering given by $(D,b_1,...,b_k)$. Then $\Gamma(B/b_k) \cong \Gamma(B)|_{U_k}$. The map on the vertices is given by $v(b_i') \mapsto v(b_i)$.\label{contprop}
\end{prop}

\begin{proof}
$\Gamma(B)|_{U_k}$ is a flag complex with vertex set $v(b_{u_1}), v(b_{u_2}),...,v(b_{u_{|U_k|}})$ and $\Gamma(B/b_k)$ is a flag complex with vertex set $v(b_{u_1}'), v(b_{u_2}'),...,v(b_{u_{|U_k|}}')$. Suppose that $i<j$ where $i,j \in U_k$. We need to show that $\{v(b_j'),v(b_{i}')\} \in \Gamma(B/b_k)$ if and only if $\{v(b_j),v(b_{i})\} \in \Gamma(B)|_{U_k}$.  We will show the following:

\begin{enumerate}
\item If $b_{i} \subseteq b_j$ (by Claim \ref{incprop}, equivalently $b_{i}' \subseteq b_j'$) then $\{v(b_j'),v(b_{i}')\} \in \Gamma(B/b_k)$ if and only if $\{v(b_j),v(b_{i})\} \in \Gamma(B)|_{U_k}$.\\
\item If $b_{i} \not \subseteq b_j$ (by Claim \ref{incprop}, equivalently $b_{i}' \not \subseteq b_j'$) then $\{v(b_j'),v(b_{i}')\} \in \Gamma(B/b_k)$ if and only if $\{v(b_j),v(b_{i})\} \in \Gamma(B)|_{U_k}$.\\
\end{enumerate}

(1) $\Rightarrow:$ Suppose that $\{v(b_j'),v(b_{i}')\} \in \Gamma(B/b_k)$, so that there exists $b \in B_{i-1}$ such that $b_{i}' \subsetneq b' \subsetneq b_j'$. By Claim \ref{incprop}, $b \subseteq b_j$ and since $b_{i} \subseteq b_j$ this implies $b \cup b_{i} \subseteq b_j$. Since $b \cap b_{i} \ne \emptyset$ we have $b \cup b_{i} \in B_{i-1}$. Hence $b_{i} \subsetneq b \cup b_{i} \subsetneq b_j$ which implies $\{v(b_{i}),v(b_j)\} \in \Gamma(B)|_{U_k}$.\\

$\Leftarrow:$ Suppose $\{v(b_{i}),v(b_j)\} \in \Gamma(B)|_{U_k}$, so that there exists $b \in B_{i-1}$ such that $b_{i} \subsetneq b \subsetneq b_j$. Then $b_{i}' \subseteq b' \subseteq b_j'$, and $b' \ne b_{i}'$ or $b_j'$ since $i, j \in U_k$, so that $b_{i}' \subsetneq b' \subsetneq b_j'$. Hence $\{v(b_{i}'),v(b_j')\} \in \Gamma(B/b_k)$.\\

(2) $\Rightarrow:$ Suppose that $\{v(b_{i}'),v(b_j')\} \in \Gamma(B/b_k)$, and suppose for a contradiction that $\{v(b_{i}),v(b_j)\} \not \in \Gamma(B)|_{U_k}$. Then there exists $b \in B_{i-1}$ such that $b \backslash b_j = b_{i} \backslash b_j$. Then $b' \backslash b_j' = b_{i}' \backslash b_j'$ which implies the contradiction that $\{v(b_{i}'),v(b_j')\} \not \in \Gamma(B/b_k)$.\\

$\Leftarrow:$ We will prove the contrapositive that $\{v(b_{i}'),v(b_j')\} \not \in \Gamma(B/b_k)$ implies that $\{v(b_{i}),v(b_j)\} \not \in \Gamma(B)|_{U_k}$. $\{v(b_{i}'),v(b_j')\} \not \in \Gamma(B/b_k)$ implies there exists $m \in B_{i-1}$ such that $m'\backslash b_j' = b_{i}'\backslash b_j'$. \\

\begin{itemize}
\item Assume that $m \subseteq b_{i}$, and for this case refer to Figure \ref{contrsubset}. Let $R: = b_k \cap (b_{i} \backslash (m \cup b_j))$, and let $J: = b_{i} \backslash (m \cup b_k)$. Since $m \subseteq b_{i}$, by Lemma \ref{decomp} there exists a decomposition of $b_{i}$ in $B_{i}$ that contains $m$. Hence $m$ is contained in a maximal component $d$ of this decomposition. Let $d'$ be the other maximal component. If $d' \cap R = \emptyset$ then $\{v(b_{i}),v(b_j)\} \not \in \Gamma(B)|_{U_k}$ since $d \backslash b_j = b_{i} \backslash b_j$, hence the desired condition holds. If $d' \cap J = \emptyset$ then $b_{i} \backslash b_k = d \backslash b_k$ which contradicts $i \in V_k$. If $d' \cap J \ne \emptyset$ and $d' \cap R \ne \emptyset$ then $(d' \cup b_j) \backslash b_k = b_j \backslash b_k$ which contradicts $j \in V_k$.\\

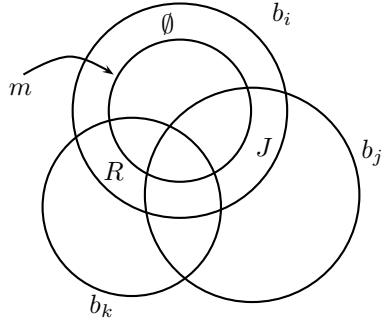
\begin{figure}[H]
\caption{A picture of the sets in case (2), assuming $m \subseteq b_{i}$. Note that $b_{i}\backslash (m \cup b_j \cup b_k) = \emptyset$ by the definition of $m$.}
\label{contrsubset}

\[
\psset{unit=0.8cm}
\begin{pspicture}(4,0)(8,5.5)
\pscircle(5.8,3.4){1.8}\rput(7.5,5){$b_{i}$}
\pscurve{->}(3.2,4)(3.975,4.3)(4.7,4)
\pscircle(5.8,3.4){1.2}\rput(3.15,3.75){$m$}
\pscircle(7,2){1.8}\rput(9,2.7){$b_j$}
\pscircle(5,1.8){1.5}\rput(4.5,.15){$b_k$}
\rput(5.6,4.85){$\emptyset$}
\rput(7.2,2.8){$J$}
\rput(4.7,2.4){$R$}
\end{pspicture}
\]
\end{figure}

\item Assume that $m \not \subseteq b_{i}$. For this case refer to Figure \ref{contrnotsubset}. Let $H : =b_{i} \backslash (b_j \cup b_k)$. In $(B_j/b_k)/b_j'$ both $b_{i}'$ and $m'$ have the same image that is given by $H$, and $H \ne \emptyset$ since $H = \emptyset$ implies $b_{i}' \subseteq b_j'$. Let $K : = m\backslash (b_k \cup b_{i})$. Then $K \ne \emptyset$ since $K = \emptyset$ implies $b_{i} \backslash b_k = m  \backslash b_k$, which contradicts $i \in V_k$. Let $L : =m \backslash (b_{i} \cup b_j)$. $L = \emptyset$ implies $\{v(b_{i}),v(b_j)\} \not \in \Gamma(B)|_{U_k}$ since $m \backslash b_j = b_{i} \backslash b_j$, so the desired condition holds. Suppose now $L \ne \emptyset$. Then $m$ intersects each of $H,K$ and $L$. Let $b$ be a minimal (for inclusion) element in in $B_{i-1}$ that intersects $H,K$ and $L$. Then $|b| \ge 3$ and at least one of the elements in the decomposition of $b$ (in $B_{i-1}$) must intersect exactly two of $K,H$ and $L$. Denote such an element by $\hat d$. If $ \hat d$ intersects $K$ and $L$ then $(b_j \cup \hat d) \backslash b_k = b_j \backslash b_k$ which contradicts $j \in V_k$. If $\hat d$ intersects both $K$ and $H$ then $\{v(b_{i}),v(b_j)\} \not \in \Gamma(B)|_{U_k}$ since $(b_{i} \cup \hat d) \backslash b_j = b_{i} \backslash b_j$, so the desired condition holds. If $\hat d$ intersects $L$ and $H$ then $(b_{i} \cup \hat d) \backslash b_k =b_{i} \backslash b_k$, which contradicts $i \in V_k$.

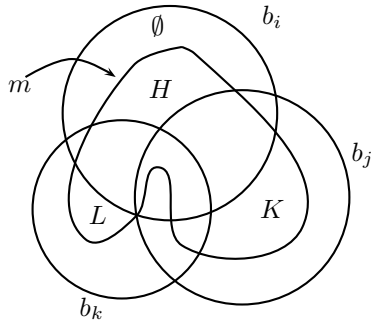
\begin{figure}[H]
\caption{A picture of the sets in case (2), assuming $m \not \subseteq b_{i}$. Note that $b_{i} \backslash ( m \cup b_j \cup  b_k) = \emptyset$ by the definition of $m$.}
\label{contrnotsubset}

\[
\psset{unit=0.8cm}
\begin{pspicture}(4,0)(8,5.5)
\pscircle(5.8,3.4){1.8}\rput(7.5,5){$b_{i}$}
\pscurve(6,4.5)(6.2,4.4)(8,1.5)(6,1.2)(5.6,2.5)(5.25,1.7)(4.4,1.3)(5.2,4.2)(6,4.5)\rput(3.3,3.85){$m$}
\pscurve{->}(3.4,4)(4.175,4.3)(4.9,4)
\rput(5.65,3.8){$H$}
\rput(4.6,1.7){$L$}\rput(7.5,1.8){$K$}
\pscircle(7,2){1.8}\rput(9,2.7){$b_j$}
\pscircle(5,1.8){1.5}\rput(4.5,.15){$b_k$}
\rput(5.6,4.85){$\emptyset$}
\end{pspicture}
\]
\end{figure}

\end{itemize}

\end{proof}

We now consider the flag building set $B|_{b_k}$. It is not necessarily true that $D|_{b_k}$ is a decomposition of $b_k$. Let $$D_k: = D|_{b_k} \cup \{b_j ~|~b_j \subseteq b_k, j \not \in V_k\}.$$ Then $D_k$ is a decomposition of $b_k$ in $B$, and for any $j$ we have that $D_k \cup \{b_i~|~i \le j ~\hbox{and}~ i \in V_k \}$ is a connected flag building set on $b_k$. We define $\Gamma(B|_{b_k})$ to be the flag complex $\Gamma(O)$ with respect to the flag ordering $O$ of $B|_{b_k}$ with decomposition $D_k$ and ordering of $B|_{b_k} - D_k$ given by $b_{v_1},b_{v_2},...,b_{u_{|V_k|}}$ where $v_j$ is the $j$th element of $V_k$ listed in increasing order. We label the vertices of $\Gamma(B|_{b_k})$ by $v(b_{v_1}),...,v(b_{u_{|V_k|}})$ rather than by their index in $V_k$. In keeping with the notation that $B_j$ is the \emph{flag} building set obtained after adding elements indexed up to $j$, we let $(B|_{b_k})_j$ denote the flag building set $D_k \cup\{b_i~|~ i \le j ~\hbox{and}~ i \in V_k\}$, so that $\Gamma((B|_{b_k})_j)$ is defined. Note then that for any $j$, $B_j|_{b_k} \subseteq (B|_{b_k})_j$.\\

\begin{prop}
Let $B$ be a connected flag building set with flag ordering given by $(D,b_1,...,b_k)$. Then $\Gamma(B|_{b_k}) = \Gamma(B)|_{V_k}$. \label{resprop}
\end{prop}

\begin{proof}
Both $\Gamma(B|_{b_k})$ and $\Gamma(B)|_{V_k}$ are both flag complexes with the vertex set $v(b_{v_1}),v(b_{v_2}),...,v(b_{u_{|V_k|}})$. We need to show that for any $i, j \in V_k$ where $i<j$, $\{v(b_{i}),v(b_j)\} \in \Gamma(B)|_{V_k}$ if and only if $\{v(b_{i}),v(b_j)\} \in \Gamma(B|_{b_k}).$\\

$\Rightarrow:$ Suppose that $\{v(b_{i}),v(b_j)\} \in \Gamma(B)|_{V_k}$. First assume that $b_{i} \subseteq b_{j}$. Then there is some $b \in B_{i-1}$ such that $b_i \subsetneq b \subsetneq b_j$. Since $b \in B_{i-1}|_{b_k}$ and $B_{i-1}|_{b_k} \subseteq (B|_{b_k})_{i-1}$ this implies that $\{v(b_i),v(b_j)\} \in \Gamma(B|_{b_k})$.\\

Now suppose that $b_i \not \subseteq b_j$. Suppose for a contradiction that $\{v(b_i),v(b_j)\} \not \in \Gamma(B|_{b_k})$. Then there exists some $d \in D_k-D|_{b_k}$, $d \not \in B_{i-1}$, such that $d \cup b_j = b_i \cup b_j$. Since $i \in V_k$ there exists some $b \in B_{i-1}$ such that $b_i \subsetneq b \subsetneq b_k$. Since $\{v(b_i),v(b_j)\} \in \Gamma(B)|_{V_k}$ we have that $b \backslash ( b_i \cup b_j) \ne \emptyset$. Since the index of $d$ is not in $V_k$, every element in the restriction to $b_k$ that is earlier than $d$ in the flag ordering is a subset of it or does not intersect it. This implies $b \subseteq d$, so $d \backslash (b_i \cup b_j) \ne \emptyset$, which contradicts $d \cup b_j \ne b_i \cup b_j$.\\

$\Leftarrow:$ Suppose that $\{v(b_{i}),v(b_j)\} \in \Gamma(B|_{b_k})$. First assume that $b_{i} \subseteq b_{j}$, so that there is some $d \in (B|_{b_k})_{i-1}$ such that $b_i \subsetneq d \subsetneq b_j$. If $d \in B_{i-1}|_{b_k}$ then clearly $\{v(b_i),v(b_j)\} \in \Gamma(B)|_{V_k}$ as desired. If $d \not \in B_{i-1}|_{b_k}$ then $d \in D_k - D|_{b_k}$. Since $i \in V_k$ there exists some $b \in B_{i-1}$ such that $b_i \subsetneq b \subsetneq b_k$. Since the index of $d$ is not in $V_k$ we have that $b_i \subsetneq b \subsetneq d$. This is because $d$ either contains or does not intersect elements that are earlier in the flag ordering and contained in $b_k$. Then since $d \subsetneq b_j$ this implies $b \subsetneq b_j$ and since $b \in B_{i-1}$ and $b_i \subsetneq b \subsetneq b_j$ this implies $\{v(b_i),v(b_j)\} \in \Gamma(B)|_{V_k}$.\\

Now assume that $b_i \not \subseteq b_j$. Suppose for a contradiction that $\{v(b_i),v(b_j)\} \not \in \Gamma(B)|_{V_k}$. Then there exists $b \in B_{i-1}|_{b_k}$ such that $b \cup b_j = b_i \cup b_j$. Since $B_{i-1}|_{b_k} \subseteq (B|_{b_k})_{i-1}$ this contradicts $\{v(b_{i}),v(b_j)\} \in \Gamma(B|_{b_k})$.

\end{proof}

\begin{thm}
Let $B$ be a connected flag building set with flag ordering $O$. Then $\gamma(B) = f(\Gamma(O))$.\label{fvectthm}
\end{thm}

\begin{proof}
This is a proof by induction on the number of elements of $B-D$. The result holds for $k=0$ since $f(\Gamma(D)) = 1 = \gamma(D)$. So we assume $k \ge 1$ and that the result holds for all connected flag building sets with a smaller value of $k$.\\

By Propositions \ref{contprop} and \ref{resprop} and the inductive hypothesis we have $f(\Gamma(B)|_{U_k}) = f(\Gamma(B/{b_k}))= \gamma(B/b_k)$, and $f(\Gamma(B)|_{V_k}) = f(\Gamma(B|_{b_k}))= \gamma(B|_{b_k})$.\\

Suppose that $u \in U_k$ and $w \in V_k$. Then $\{v(b_u),v(b_w)\} \in \Gamma(B)$, for suppose for a contradiction that $\{v(b_u),v(b_w)\} \not \in \Gamma(B)$. Suppose that $u <w$. Then there is some element $b \in B_{u-1}$ such that $b \cup b_w = b_u \cup b_w$. This implies that $b \cup b_k = b_u \cup b_k$ which contradicts $ u \in U_k$. Suppose that $w < u$. Then either $b_u \cap b_w = \emptyset$ or $b_w \subseteq b_u$ (otherwise $b_u \cup b_w$ makes $b_u$ degenerate with respect to $b_k$). Suppose that $b_w \cap b_u = \emptyset$. Then since $\{v(b_u),v(b_w)\} \not \in \Gamma(B)$, there exists $b \in B_{w-1}$ such that $b \cup b_u = b_w \cup b_u$, and $b \cap b_u \ne \emptyset$. Then $b \cup b_u$ makes $b_u$ degenerate with respect to $b_k$, a contradiction. Suppose that $b_w \subseteq b_u$. Now $w \in V_k$ implies there is some $b \in B_{w-1}$ such that $b_w \subsetneq b \subsetneq b_k$. Also, $b \subseteq b_u$ else $b \cup b_u$ makes $b_u$ degenerate with respect to $b_k$. However, this implies the contradiction that $\{v(b_u),v(b_w)\} \in \Gamma(B)$ since $b_w \subsetneq b \subsetneq b_u$.\\

Hence $$\Gamma(B)|_{U_k \cup V_k} = \Gamma(B)|_{U_k}*\Gamma(B)|_{V_k},$$ and therefore $$f(\Gamma(B)|_{U_k \cup V_k}) = f(\Gamma(B)|_{U_k})f(\Gamma(B)|_{V_k})=\gamma(B/b_k)\gamma(B|_{b_k}).$$ Since the vertex $v(b_k)$ is adjacent to the vertices indexed by elements in $U_k \cup V_k$ we have
$$f(\Gamma(B)) = f(\Gamma(B_{k-1})) + t\gamma(B/b_k)\gamma(B|_{b_k}).$$ By the induction hypothesis this implies that $$f(\Gamma(B)) =\gamma(B_{k-1}) + t\gamma(B|{b_k})\gamma(B/b_k),$$ which implies that $f(\Gamma(B)) = \gamma(B)$ by Theorem \ref{vollem}.
\end{proof}

For two flag orderings $O_1$, $O_2$ of a connected flag building set $B$, it is not necessarily true that the flag complexes $\Gamma(O_1)$, $\Gamma(O_2)$ are equivalent (up to change of labels on the vertices) even if they have the same decomposition. The following example provides a counterexample.

\begin{example}
Let $B = B(Cyc_5)$, and let $$D =\{\{1\},\{2\},\{3\},\{4\},\{5\},[2],[3],[4],[5]\}.$$ Let $O_1$ be the flag ordering with decomposition $D$ and the following ordering of $B-D$:

$$\{2,3\},\{2,3,4\},\{2,3,4,5\},\{4,5\},\{3,4,5\},\{3,4\},$$ $$\{3,4,5,1\},\{4,5,1,2\},\{5,1,2,3\},\{4,5,1\},\{5,1,2\},\{1,5\}.$$ Let $O_2$ be the flag ordering with decomposition $D$ and the following ordering of $B-D$:

$$\{2,3\},\{2,3,4\},\{2,3,4,5\},\{3,4\},\{3,4,5\},\{4,5\},\{3,4,5\},\{3,4\},$$ $$\{3,4,5,1\},\{4,5,1,2\},\{5,1,2,3\},\{4,5,1\},\{5,1,2\},\{1,5\}.$$

Then $\Gamma(O_1)$ and $\Gamma(O_2)$ are depicted in Figure \ref{cycfive}.

\begin{figure}[H]
\caption{$\Gamma(O_1)$ is on the left, and $\Gamma(O_2)$ is on the right.}
\label{cycfive}

\[
\psset{unit=0.8cm}
\begin{pspicture}(-.7,-3)(3,3)
\qdisk(0,2.5){2.25pt}\qdisk(1.25,2.165){2.25pt}\qdisk(2.165,1.25){2.25pt}\qdisk(2.5,0){2.25pt}
\qdisk(2.165,-1.25){2.25pt}\qdisk(1.25,-2.165){2.25pt}\qdisk(0,-2.5){2.25pt}\qdisk(-1.25,-2.165){2.25pt}
\qdisk(-2.165,-1.25){2.25pt}\qdisk(-2.5,0){2.25pt}\qdisk(-2.165,1.25){2.25pt}\qdisk(-1.25,2.165){2.25pt}
\rput(0,2.8){$v(b_1)$}\rput(1.6,2.5){$v(b_2)$}\rput(2.7,1.5){$v(b_3)$}\rput(3.2,0){$v(b_4)$}
\rput(2.7,-1.5){$v(b_5)$}\rput(1.4,-2.6){$v(b_6)$}\rput(0,-2.9){$v(b_7)$}\rput(-1.4,-2.55){$v(b_8)$}
\rput(-2.6,-1.6){$v(b_9)$}\rput(-3.2,0){$v(b_{10})$}\rput(-2.8,1.5){$v(b_{11})$}\rput(-1.6,2.5){$v(b_{12})$}
\psline(2.5,0)(0,2.5)(-2.165,-1.25)\psline(2.165,1.25)(1.25,-2.165)(0,-2.5)\psline(2.165,-1.25)(-1.25,2.165)(-1.25,-2.165)
\end{pspicture}
\begin{pspicture}(-4.8,-3)(0,3)
\qdisk(0,2.5){2.25pt}\qdisk(1.25,2.165){2.25pt}\qdisk(2.165,1.25){2.25pt}\qdisk(2.5,0){2.25pt}
\qdisk(2.165,-1.25){2.25pt}\qdisk(1.25,-2.165){2.25pt}\qdisk(0,-2.5){2.25pt}\qdisk(-1.25,-2.165){2.25pt}
\qdisk(-2.165,-1.25){2.25pt}\qdisk(-2.5,0){2.25pt}\qdisk(-2.165,1.25){2.25pt}\qdisk(-1.25,2.165){2.25pt}
\rput(0,2.8){$v(b_1)$}\rput(1.6,2.5){$v(b_2)$}\rput(2.7,1.5){$v(b_3)$}\rput(3.2,0){$v(b_4)$}
\rput(2.7,-1.5){$v(b_5)$}\rput(1.4,-2.6){$v(b_6)$}\rput(0,-2.9){$v(b_7)$}\rput(-1.4,-2.55){$v(b_8)$}
\rput(-2.6,-1.6){$v(b_9)$}\rput(-3.2,0){$v(b_{10})$}\rput(-2.8,1.5){$v(b_{11})$}\rput(-1.6,2.5){$v(b_{12})$}
\psline(-2.165,-1.25)(0,2.5)(1.25,-2.165)(0,-2.5)\psline(-1.25,-2.165)(-1.25,2.165)(2.5,0)(2.165,1.25)
\end{pspicture}
\]
\end{figure}
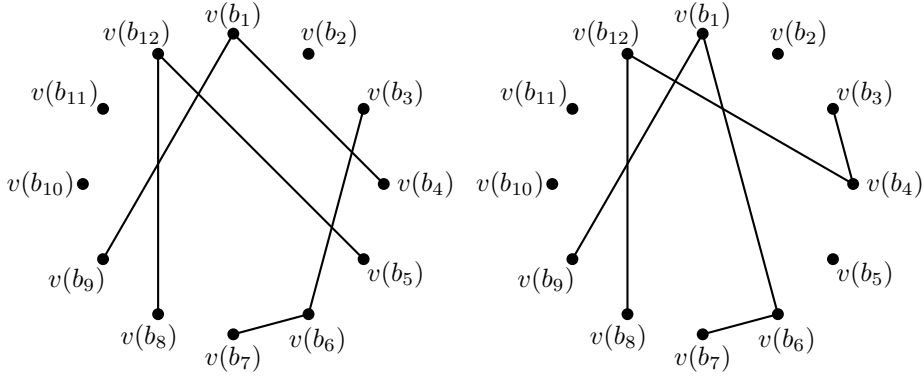

\end{example}

\end{section}

\begin{section}{The flag complexes of Nevo and Petersen}

In this section we compare the flag complexes that we have defined to those defined for certain graph-associahedra by Nevo and Petersen \cite{np}. They define flag complexes $\Gamma(\widehat{\mathfrak{S}}_n)$, $\Gamma(\widehat{\mathfrak{S}}_n(312))$ and $\Gamma(P_n)$ such that

\begin{itemize}
\item $\gamma(B(K_n)) = f(\Gamma(\widehat{\mathfrak{S}}_n)),$
\item $\gamma(B(Path_n)) = f(\Gamma(\widehat{\mathfrak{S}}_n(312))),$
\item $\gamma(B(Cyc_n)) = f(\Gamma(P_n)).$
\end{itemize}
We show that there is a flag ordering for $B(Path_n)$ so that $$\Gamma(B(Path_n)) \cong \Gamma(\widehat{\mathfrak{S}}_n(312)),$$ and that the analogous statement is not true for $B(K_n)$ and $B(Cyc_n)$.  \\

\begin{subsection}{The flag complexes $\Gamma(B(K_n))$ and $\Gamma(\widehat{\mathfrak{S}}_n)$}

The permutohedron is the nestohedron $P_{B(K_n)}$. Note that $B(K_n)$ consists of all nonempty subsets of $[n]$. The $\gamma$-polynomial of $P_{B(K_n)}$ is the descent generating function of $\widehat{\mathfrak{S}}_n,$ which denotes the set of permutations with no double descents or final descent (see \cite[Theorem 11.1]{prw}). First we recall the definition of $\Gamma(\widehat{\mathfrak{S}}_n)$ given by Nevo and Petersen \cite[Section 4.1]{np}.\\

A \emph{peak} of a permutation $w=  w_1....w_n$ in $\mathfrak{S}_n$ is a position $i \in [1,n-1]$ such that $w_{i-1} < w_i >w_{i+1}$, (where $w_0 :=0$). We denote a peak at position $i$ with a bar $w_1..w_i|w_{i+1}...w_n$. A \emph{descent} of a permutation $w= w_1...w_n$ is a position $i \in [n-1]$ such that $w_{i+1}< w_i$. Let $\widehat{\mathfrak{S}}_n$ denote the set of permutations in $\mathfrak{S}_n$ with no double (i.e. consecutive) descents or final descent, and let $\widetilde{\mathfrak{S}}_n$ denote the set of permutations in $\mathfrak{S}_n$ with one peak. Then $\widehat{\mathfrak{S}}_n \cap \widetilde{\mathfrak{S}}_n$ consists of all permutations of the form

$$w_1...w_i|w_{i+1}...w_n$$ where $~1 \le i \le n-2,~$ $~w_1< \cdots < w_i,~$ $~w_i > w_{i+1},~$ $~w_{i+1}< \cdots <w_n.$\\

Define the flag complex $\Gamma(\widehat{\mathfrak{S}}_n)$ on the vertex set $\widehat{\mathfrak{S}}_n \cap \widetilde{\mathfrak{S}}_n$ where two vertices

$$u =u_1|u_2$$ and
$$v = v_1|v_2$$ with $|u_1| < |v_1|$ are adjacent if there is a permutation $w \in \mathfrak{S}_n$ of the form

$$w = u_1|a|v_2.$$ Equivalently, if $v_2 \subseteq u_2$, $|u_2 - v_2| \ge 2,$ $min(u_2- v_2)< max(u_1)$ and $max(u_2 - v_2) >min(v_2)$. (Since there must be two peaks in $w$ this implies $|a| \ge 2$). The faces of $\Gamma(\widehat{\mathfrak{S}}_n)$ are the cliques in this graph.\\

\begin{example}
Taking only the part after the peak, $\widehat{\mathfrak{S}}_5 \cap \widetilde{\mathfrak{S}}_5$ can be identified with the set of subsets of $[5]$ of sizes 2,3 and 4 which are not $\{4,5\},\{3,4,5\}$, or $\{2,3,4,5\}$. Then the edges of $\Gamma(\widehat{\mathfrak{S}}_5)$ are given by:\\
$\{1,2,3,4\}$ is adjacent to each of $\{1,2\},\{1,3\},\{1,4\},\{2,3\},\{2,4\}$,\\
$\{1,2,3,5\}$ is adjacent to each of $\{1,2\},\{1,3\},\{1,5\},\{2,3\},\{2,5\}$,\\
$\{1,2,4,5\}$ is adjacent to each of $\{1,4\},\{1,5\},\{2,4\},\{2,5\}$, and \\
$\{1,3,4,5\}$ is adjacent to each of $\{3,4\},\{3,5\}$.
\end{example}

\begin{prop}
There is no flag ordering of $B(K_5)$ so that
$$\Gamma(B(K_5)) \cong \Gamma(\widehat{\mathfrak{S}}_5).$$
\end{prop}

\begin{proof}
Suppose for a contradiction that there is some flag ordering of $B(K_5)$ with decomposition $D$ such that $\Gamma(B(K_5)) \cong \Gamma(\widehat{\mathfrak{S}}_5)$. Then there is some vertex $v(b_j) \in \Gamma(B(K_5))$ of degree 5. We consider the following three cases:

\begin{enumerate}
\item $|b_j|=2$,
\item $|b_j| =3$,
\item $|b_j|=4$.
\end{enumerate}
Note that $D$ can only be one of the following three building sets (up to the order reversing permutation of $B$):

$$\{\{1\},...,\{5\},[2],[3],[4],[5]\},$$ $$\{\{1\},...,\{5\},\{1,2\},\{3,4\},[4],[5]\},$$ $$\{\{1\},...,\{5\},[2],[3],\{4,5\},[5]\}.$$

\begin{enumerate}
\item[(1)] Suppose that $|b_j| =2$. Then $V_j = \emptyset$ and $|U_j| \le 2$ (using the fact that $D/b_j$  includes at least one 2-element subset). So there are $\ge 3$ $b_k$'s with $k>j$ and $j \in U_k \cup V_k$ (i.e. $v(b_j)$ is adjacent to $v(b_k)$). Such $b_k$'s must be two element sets not intersecting $b_j$ or four element sets that contain $b_j$. Without loss of generality (WLOG for short), let $b_j = \{4,5\}$.
    \begin{enumerate}
    \item[(1a)] Suppose that no three element set containing $b_j$ occurs earlier than $b_j$. Then the case of 4-element $b_k$'s cannot occur, so $\{1,2\},\{1,3\},\{2,3\}$ are the $b_k$'s. Since there is a 2-element set in $D$, we have (WLOG) $\{3,5\} \in D$, implying that $\{3,4,5\}$ is earlier than $\{4,5\}$, a contradiction.
    \item[(1b)]Suppose that exactly one 3-element set containing $b_j$, WLOG $\{3,4,5\}$, occurs earlier than $b_j$. Then $\{1,3\},\{2,3\},\{1,2,4,5\}$ can't occur among the $b_k$'s, so $\{1,2\},\{1,3,4,5\},\{2,3,4,5\}$ are the $b_k$'s. Hence $|U_j| =2$, so $B_{j-1}/ b_j$ consists of all non-empty subsets of $\{1,2,3\}$. So $\exists b \in B_{j-1}$ such that $b \backslash b_j = \{2,3\}$. But then $b \in B_{j-1}$, $\{3,4,5\} \in B_{j-1}$ implies $b \cup \{2,3,4\} = \{2,3,4,5\} \in B_{j-1}$, a contradiction.
    \item[(1c)] Suppose that there are at least two 3-element sets containing $b_j$, WLOG $\{2,4,5\}$ and $\{3,4,5\}$, that occur earlier than $b_j$. Then $\{2,3,4,5\}$ occurs earlier than $b_j$ and $\{1,2\},\{1,3\},\{2,3\}$ can't occur among the $b_k$'s, so we have a contradiction.
    \end{enumerate}
    ~\\
\item[(2)] Suppose that $|b_j| = 3.$ It is easy to see that $v(b_j)$ is not adjacent to any vertices $v(b_i)$ where $i<j$, i.e. $U_j =V_j = \emptyset$. Hence there must be 5 elements $b_k$, $k>j$, such that $j \in V_k \cup U_k$, and these elements must be of size 2. Suppose WLOG, that $b_j = \{1,2,3\}$, and that the five elements $b_k$ are $$\{1,4\},\{2,4\},\{3,4\},\{1,5\},\{2,5\}.$$ There is one two element subset of $b_j$ that is earlier than $b_j$ in the flag ordering since $b_j$ requires a decomposition, and this element must have the same image in the contraction by one of $\{1,4\},\{2,4\},\{3,4\},\{1,5\},\{2,5\}$ as $\{1,2,3\}$, hence this case cannot occur.\\

\item[(3)] Suppose that $|b_j| =4$. Note that $U_j = \emptyset$.
\begin{itemize}
\item[(3a)] Suppose that no three element subset of $b_j$ occurs earlier than $b_j$. Then $V_j = \emptyset$, so there are at least five $b_k$ $k>j$ such that $j \in U_k \cup V_k$. These $b_k$'s are clearly 2-element subsets of $b_j$, but for $b_j$ to have a decomposition in $B_j$, two of the 2-element subsets of $b_j$ must occur earlier than $b_j$, a contradiction.
\item[(3b)] Suppose WLOG that $b_j = \{1,2,3,4\}$ and that $\{1,2,3\}$ occurs earlier than $b_j$. Since $B_{j-1}$ is a building set no other 3-element subset of $b_j$ occurs before $b_j$. If $v(b_j)$ is adjacent to $v(b_k)$ then either $k<j$ which forces $b_k$ to be a 2-element subset of $\{1,2,3\}$, or $k>j$ which also forces $b_k$ to be a two element subset of $\{1,2,3\}$ (so that $\{1,2,3\} \backslash b_k \ne \{1,2,3,4\} \backslash b_k$ and $\{1,2,3,4,5\} \backslash b_k \ne \{1,2,3,4\} \backslash b_k$). So $v(b_j)$ is adjacent to at most three vertices, a contradiction.
\end{itemize}
\end{enumerate}

Since we have shown that none of the cases (1), (2) or (3) can occur we have a contradiction, as desired.
\end{proof}

We will now give a combinatorial description of $\Gamma(B(K_n))$ for a particular flag ordering. Let $O$ be the flag ordering of $B=B(K_n)$ with decomposition

$$D=\{\{1\},\{2\},...,\{n\},[2],[3],...,[n]\}$$ where elements $a,b \in B-D$ are ordered so that $a$ is earlier than $b$ if:

\begin{itemize}
\item $\max(a) < \max(b),$ or
\item $\max(a) = \max(b)$ and $|a|>|b|,$ or
\item $\max(a) = \max(b)$, $|a| = |b|$ and $\min(a \nabla b) \in a$
\end{itemize}
where $\nabla$ denotes the symmetric difference between two sets.\\

Then in $\Gamma(O)$, vertices corresponding to elements $a,b \in B-D$ are adjacent if either:
\begin{itemize}
\item $a \subseteq b$ and $\min(b-a) < \max(a),$
\item $\max(a) \not \in b$ and $|a \backslash b| \ge 2$ and $\min(b \backslash a) >\max(a)$.
\end{itemize}

\begin{example}
The edges of $\Gamma(B(K_5))$ are between the consecutive vertices in the following three sequences, which form cycles:

$$v(\{1,4\}),v(\{1,2,4,5\}),v(\{2,4\}),v(\{2,3,4,5\}),v(\{3,4\}),v(\{1,3,4,5\}),v(\{1,4\})$$ and
$$v(\{1,3\}),v(\{1,2,3,5\}),v(\{2,3\}),v(\{4,5\}),v(\{1,3\})$$ and
$$v(\{1,2,4\}),v(\{1,5\}),v(\{1,3,4\}),v(\{3,5\}),v(\{2,3,4\}),v(\{2,5\}),v(\{1,2,4\}).$$
\end{example}

\end{subsection}

\begin{subsection}{The flag complexes $\Gamma(B(Path_n))$ and $\Gamma(\widehat{\mathfrak{S}}_n(312))$}

The associahedron is the nestohedron $P_{B(Path_n)}.$ Note that $B(Path_n)$ consists of all intervals $[j,k]$ with $1 \le j \le k \le n$. The $\gamma$-polynomial of the associahedron is the descent generating function of $\widehat{\mathfrak{S}}_n(312)$, which denotes the set of 312-avoiding permutations with no double or final descents (see \cite[Section 10.2]{prw}). We now describe the flag complex $\Gamma(\widehat{\mathfrak{S}}_n(312))$ defined by Nevo and Petersen \cite[Section 4.2]{np}.\\

Given distinct integers $a,b,c,d$ such that $a<b$ and $c<d$, the pairs $(a,b),(c,d)$ are \emph{non-crossing} if either

\begin{itemize}
\item $a<c<d<b$ (or $c<a<b<d$), or
\item $a<b<c<d$ (or $c<d<a<b).$
\end{itemize}

Define $\Gamma(\widehat{\mathfrak{S}}_n(312))$ to be the flag complex on the vertex set
$$V_{n}: = \{(a,b)~|~1 \le a<b \le n-1\},$$ with faces the sets $S$ of $V_n$ such that if $(a,b) \in S$ and $(c,d) \in S$ then $(a,b)$ and $(c,d)$ are non-crossing.\\

Let $O$ denote the flag ordering of $B =B(Path_n)$ with decomposition
$D=\{\{1\},\{2\},\{3\},\{4\},\{5\},[2],[3],[4],[5]\},$ where elements $a,b \in B-D$ are ordered so that $a$ is earlier than $b$ if:

\begin{itemize}
\item $\max(a) < \max(b)$, or
\item $\max(a) = \max(b)$ and $|a| > |b|$.
\end{itemize}

\begin{prop}
For the flag ordering $O$ of $B = B(Path_n)$ described above, $\Gamma(O) \cong \Gamma(\widehat{\mathfrak{S}}_n(312))$ where the bijection on the vertices is given by $v([a+1,b+1]) \mapsto (a,b)$.
\end{prop}

\begin{proof}
Since $B-D =\{[j,k]~|~2 \le j < k \le n\},$ it is clear that the stated map on vertices is a bijection. Let $[l,m],[j,k]$ be distinct elements of $B-D$ with $[l,m]$ occurring before $[j,k]$. Then $m \le k$, and if $m=k$ we have $l<j$. If $[l,m] \not \subseteq [j,k]$ then $v([l,m])$ is adjacent to $v([j,k])$ if and only if $m<j$. If $[l,m] \subseteq [j,k]$ (which entails $m<k$), then $v([l,m])$ is adjacent to $v([j,k])$ if and only if $j<l$. So in either case $v([l,m])$ is adjacent to $v([j,k])$ if and only if $(l-1,m-1)$ and $(j-1,k-1)$ are non-crossing.

\end{proof}

\end{subsection}

\begin{subsection}{The flag complexes $\Gamma(B(Cyc_n))$ and $\Gamma(P_n)$}

The cyclohedron is the nestohedron $P_{B(Cyc_n)}.$ Note that $B(Cyc_n)$ consists of all sets $\{i,i+1,i+2,...,i+s\}$ where $i \in [n]$, $s \in \{0,1,...,n-1\}$, and the elements are taken mod $n$. By \cite[Proposition 11.15]{prw} $\gamma_r(B(Cyc_n)) = \binom{n}{r,r,n-2r}$. We now describe the flag complex $\Gamma(P_n)$ defined by Nevo and Petersen \cite[Section 4.3]{np}.\\

Define the vertex set

$$V_{P_n} :=\{(l,r) \in [n-1] \times [n-1]~|~l \ne r\}.$$ $\Gamma(P_n)$ is the flag complex on the vertex set $V_{P_n}$ where vertices $(l_1,r_1),(l_2,r_2)$ are adjacent in $\Gamma(P_n)$ if and only if $l_1,l_2,r_1,r_2$ are all distinct and either $l_1<l_2$ and $r_1 <r_2$, or $l_2<l_1$ and $r_2 <r_1$.

\begin{example}
$\Gamma(P_5)$ is the flag complex on vertices $$V_{P_5} =\{(1,2),(1,3),(1,4),(2,3),(2,4),(3,4),(2,1),(3,1),(4,1),(3,2),(4,2),(4,3)\}$$ with edges

$$\{(1,3),(2,4)\},\{(3,1),(4,2)\},\{(1,2),(3,4)\},$$
$$\{(1,2),(4,3)\},\{(2,1),(4,3)\},\{(2,1),(3,4)\}.$$
\end{example}
Note that $\Gamma(P_5)$ has exactly two vertices of degree two, and has six connected components, four of which contain more than one vertex.

\begin{prop}
There is no flag ordering of $B(Cyc_5)$ so that $\Gamma(B(Cyc_5)) \cong \Gamma(P_5)$.
\end{prop}

\begin{proof}
Suppose for a contradiction that there is some flag ordering of $B=B(Cyc_5)$ with decomposition $D$ such that $\Gamma(B(Cyc_5)) \cong \Gamma(P_5)$. It is not too hard to show that if vertices $v(a)$ and $v(b)$ are adjacent then at least one of $a$ or $b$ is a 2-element set. Therefore there must be at least one vertex that corresponds to a building set element of size two in each of the four non-singleton connected components of $\Gamma(B(Cyc_5))$. Since there must be one two element subset in $D$ this implies that there is exactly one vertex corresponding to a two element set in each non-singleton connected component, and these include the vertices of degree two.\\

The possibilities for $D$ (up to a cyclic permutation of $B$) are

$$D_1=\{[5],[4],[3],[2],\{1\},...,\{5\}\},$$ $$D_2=\{[5],[2],\{5,1,2\},\{5,1,2,3\},\{1\},...,\{5\}\},$$
$$D_3=\{[5],[4],\{1,2\},\{3,4\},\{1\},...,\{5\}\},$$ $$D_4=\{[5],[3],\{4,5\},\{1,2\},\{1\},...,\{5\}\}.$$

The flag ordering must have decomposition $D_1$ or $D_2$ since there are four elements of size two in $B(Cyc_5) -D.$ We will show that if $D$ is $D_1$ or $D_2$ then there must be two vertices in $\Gamma(B(Cyc_5))$ that are adjacent that correspond to building set elements of size two, a contradiction.\\

The size two elements in $B-D_1$ and $B-D_2$ are $\{2,3\},\{3,4\},\{4,5\},\{5,1\}$. If $b_j \in B - D_i$ is earlier in the flag ordering than every one of these size two elements, then $b_j$ must contain $\{1,2\}$ since otherwise it would not have a decomposition in $B_j$. So the only elements of $B-D_1$ that can be earlier in the flag ordering than every element of size two are

$$S_1 = \{\{5,1,2,3\},\{5,1,2\},\{4,5,1,2\}\}.$$ Similarly, the elements of $B-D_2$ that can be earlier in the flag ordering than every element of size two are

$$S_2 = \{\{1,2,3\},\{4,5,1,2\},\{1,2,3,4\}\}.$$

Consider which of the size two elements in $B-D_i$ is earliest.
Suppose that $\{1,5\}$ is earliest in the flag order. Then $v(\{3,4\})$ is adjacent to $v(\{1,5\})$ since $\{1,5\} \not \in (D_1 \cup S_1)/\{3,4\} =(D_2 \cup S_2)/\{3,4\}.$\\

Suppose that $\{2,3\}$ is earliest in the flag order. Then $v(\{4,5\})$ is adjacent to $v(\{2,3\})$ since $\{2,3\} \not \in (D_1 \cup S_1)/\{4,5\} =(D_2 \cup S_2)/\{4,5\}$.\\

Suppose that $\{3,4\}$ is earliest in the flag order. Then $v(\{1,5\})$ is adjacent to $v(\{3,4\})$ since $\{3,4\} \not \in (D_1 \cup S_1)/\{1,5\} =(D_2 \cup S_2)/\{1,5\}$.\\

Suppose that $\{4,5\}$ is earliest in the flag order. Then $v(\{2,3\})$ is adjacent to $v(\{4,5\})$ since $\{4,5\} \not \in (D_1 \cup S_1)/\{2,3\} =(D_2 \cup S_2)/\{2,3\}$.\\

\end{proof}

\end{subsection}

\begin{subsection}{The flag complex $\Gamma(B(K_{1,n-1}))$.}

Here we give a combinatorial description of $\Gamma(B(K_{1,n-1}))$ for a particular flag ordering. $B=B(K_{1,n-1})$ is the graphical building set for the graph $K_{1,n-1}$ where we assume the vertex of degree $n-1$ is labelled $1$. So $B(K_{1,n-1})$ consists of all subsets of $[n]$ containing 1, together with $\{2\},\{3\},...,\{n\}$. Let $O$ be the flag ordering with decomposition $$D = \{[n],[n-1],...,[2],\{1\},\{2\},...,\{n\}\},$$ where $a,b \in B-D$ are ordered so that $a$ is earlier than $b$ if:

\begin{itemize}
\item $\max(a) < \max(b)$, or
\item $\max(a) =\max(b)$ and $|a| > |b|,$ or
\item $\max(a) = \max(b)$ and $|a| = |b|$ and $\min(a \nabla b) \in a.$
\end{itemize}

Then in $\Gamma(O)$, vertices corresponding to elements $a,b \in B-D$ are adjacent if either:
\begin{itemize}
\item $a \subseteq b$ and $\min(b-a) < \max(a),$
\item $\max(a) \not \in b$ and $|a \backslash b| \ge 2$ and $\min(b \backslash a) >\max(a)$.
\end{itemize}

\begin{example}
The edges of $\Gamma(B(K_{1,4}))$ are:

$$\{v(\{1,5\}),v(\{1,2,4\})\},\{v(\{1,5\}),v(\{1,3,4\})\}$$
$$\{v(\{1,3,4,5\}),v(\{1,4\})\},\{v(\{1,2,4,5\}),v(\{1,4\})\}.$$
\end{example}

In fact, for this flag ordering, the restriction of $\Gamma(K_{1,n-1})$ to the vertices corresponding to sets of size $\ge 3$ is isomorphic to $\Gamma(K_{n-1})$ for the flag ordering defined in Section 4.1.

\end{subsection}

\end{section}

\end{document}